\documentclass[11pt]{amsart}
\usepackage[sorted-cites]{amsrefs}
\usepackage{verbatim}
\usepackage{color}
\usepackage{amscd}
\renewcommand{\bar}{\overline}

\newcommand{\ric}{\textrm{Ric}}
\newcommand{\rr}{\mathbb{R}}
\newfont{\fnt}{cmr10 scaled 550}

\newtheorem{theorem}{Theorem}

\newtheorem{lemma}{Lemma}
\newtheorem{cor}{Corollary}
\newtheorem{prop}{Proposition}

\newtheorem{definition}{Definition}

\newtheorem{example}{Example}
\theoremstyle{remark}
\newtheorem{remark}{Remark}

\font\strange=msbm10

\renewcommand{\epsilon}{\varepsilon}

\renewcommand{\Sigma}{\varSigma}
\newcommand{\R}{{{\mathchoice  {\hbox{$\textstyle{\text{\strange R}}$}}
{\hbox{$\textstyle{\text{\strange R}}$}}
{\hbox{$\scriptstyle  N\kern-0.3em  R$}}
{\hbox{$\scriptscriptstyle  R\kern-0.2em  R$}}}}}

\newcommand{\Z}{{{\mathchoice  {\hbox{$\textstyle{\text{\strange Z}}$}}
{\hbox{$\textstyle{\text{\strange Z}}$}}
{\hbox{$\scriptstyle  Z\kern-0.3em  Z$}}
{\hbox{$\scriptscriptstyle  Z\kern-0.2em  Z$}}}}}

\newcommand{\N}{{{\mathchoice  {\hbox{$\textstyle{\text{\strange N}}$}}
{\hbox{$\textstyle{\text{\strange N}}$}}
{\hbox{$\scriptstyle  N\kern-0.3em  N$}}
{\hbox{$\scriptscriptstyle  N\kern-0.2em  N$}}}}}

\usepackage{amsmath,amsthm,amscd}

\begin{document}
\title[Eigenvalues of drifted Laplacian]{Eigenvalues of the drifted Laplacian on complete metric measure spaces}



\thanks{The authors were partially supported by CNPq and Faperj of Brazil.}

\address{Instituto de Matematica, Universidade Federal Fluminense,
Niter\'oi, RJ 24020, Brazil}
\author{Xu Cheng}
\email{xcheng@impa.br}
\author[Detang Zhou]{Detang Zhou}

\email{zhou@impa.br}

\newcommand{\M}{\mathcal M}

\begin{abstract}  In this paper, first we study  a complete smooth metric measure space $(M^n,g, e^{-f}dv)$ with  the ($\infty$)-Bakry-\'Emery Ricci curvature   $\ric_f\ge \frac a2g$ for some positive constant $a$.  It is known that the spectrum of the drifted Laplacian $\Delta_f$ for $M$ is discrete and the first nonzero eigenvalue of $\Delta_f$ has lower bound $\frac a2$. We prove that  if  the lower bound    $\frac a2$   is achieved with multiplicity $k\geq 1$, then $k\leq n$,  $M$ is isometric to $\Sigma^{n-k}\times \mathbb{R}^k$  for some complete $(n-k)$-dimensional manifold $\Sigma$ and by passing an isometry, $(M^n,g, e^{-f}dv)$ must split off a gradient shrinking Ricci  soliton $(\mathbb{R}^k, g_{can}, \frac{a}{4}|t|^2)$, $t\in \mathbb{R}^k$. This result has an application to gradient shrinking Ricci solitons. Secondly, we  study the drifted Laplacian $\mathcal{L}$ for properly immersed self-shrinkers in the Euclidean space $\mathbb{R}^{n+p}$, $p\geq1$ and show 
the discreteness of the spectrum of $\mathcal{L}$ and  a logarithmic Sobolev inequality.
\end{abstract}

\maketitle
\baselineskip=18pt
\section{Introduction}
 The well-known  Lichnerowicz theorem  \cite{L} states that if the Ricci curvature of a closed, i.e., compact and without boundary,   Riemannian manifold $(M^n,g)$ of dimension $n\geq 2$ satisfies
$\text{Ric} \geq (n-1)a$, where $a$ is a positive constant, then the first nonzero eigenvalue of the Laplacian $\Delta$ satisfies $\lambda_1 \ge na$. Obata's Theorem \cite{O} states that equality holds if and only if the manifold is an $n$-dimensional sphere with constant sectional curvature $a$. Observe that if a complete Riemannian  manifold $M$ has Ricci curvature bounded from below by a positive constant, then $M$ must be compact.

One may ask whether a phenomenon corresponding to Lichnerowicz-Obata's theorem would happen for complete smooth metric measure spaces $(M^n,g, e^{-f}dv)$ with the ($\infty$)-Bakry-\'Emery Ricci curvature   $\ric_f:=\ric+\nabla^2f\ge \frac a2g$ for some positive  constant $a$.  Recall that a complete smooth metric measure space $(M^n,g,e^{-f}dv)$   is a complete  $n$-dimensional Riemannian manifold $(M^{n}, g)$ together with a weighted volume form $e^{-f}dv$ on $M$, where $f$ is a smooth function on $M$ and $d\nu$ the volume element induced by the metric $g$. For an $(M,g,e^{-f}dv)$,  a suitable operator on $M$ is the drifted Laplacian $\Delta_f=\Delta-\langle\nabla f, \nabla \cdot\rangle$, where $\Delta$ is the Laplacian on $M$. The reason is that $\Delta_f$ is  a densely defined self-adjoint operator in the space of square-integrable functions on $M$ with respect to the measure $e^{-f} d\sigma$, that is,  for $u, w\in C_0^{\infty}(M)$,
$$\int_M u\Delta_fwe^{-f}dv=-\int_M\langle\nabla u,\nabla w\rangle e^{-f}dv.$$
For  $(M,g, e^{-f}dv)$  with  $\ric_f\ge \frac a2g$  for some constant $a>0$,    it is known that $M$  is not necessary to be compact. One of  examples is the  Gaussian shrinking soliton $(\mathbb{R}^n, g_{can}, \frac{|x|^2}{4})$ with  the canonical Euclidean metric $g_{can}$, $f=\frac{|x|^2}{4}$, $x\in \mathbb{R}^k$ and $\ric_f=\frac 12g_{can}$.  Hence one must deal with complete manifolds including noncompact case, which is  different from Lichnerowicz-Obata's theorem.  
In 1985, Bakry-\'Emery \cite{BE} showed that, if $(M^n,g, e^{-f}dv)$ has $\ric_f\ge \frac a2g$ for some  constant $a>0$ and  finite weighted volume $\int_Me^{-f}dv $, then the following  logarithmic Sobolev inequality holds:
\begin{equation}\label{varphi-1.1}
\int_M u^2\log u^2e^{-f}dv\le \frac 4a\int_M|\nabla u|^2e^{-f}dv,
\end{equation}
 for all  $u\in C_0^{\infty}(M)$ satisfying $\int_Mu^2e^{-f}dv=\int_Me^{-f}dv$.

The  logarithmic Sobolev inequality \eqref{varphi-1.1}  together with the later work of  Morgan  \cite{M} and
Hein-Naber \cite{HN} 
leads  the  following  Lichnerowicz type theorem for $\Delta_f$ (see Section \ref{sec-embed} for details):

\begin{theorem}\label{th-eigen} (Bakry-\'Emery-Morgan-Hein-Naber)  Let $(M^n, g,e^{-f}dv)$ be a complete smooth metric measure space with $\ric_f\ge \frac a2g$ for some constant  $a>0$. Then 
\begin{itemize}
\item the spectrum of $\Delta_f$ for $M$ is discrete.
\item the first nonzero eigenvalue, denoted by  $\lambda_1(\Delta_f)$,  of $\Delta_f$ for $M$ is the spectrum gap of $\Delta_f$ and satisfies
\begin{equation}\label{ine-lambda-1}
\lambda_1(\Delta_f)\ge \frac a2.
\end{equation}
\end{itemize}
\end{theorem}
In this paper,  first we study the rigidity of equality in \eqref{ine-lambda-1}. Observe that the lower bound $\frac a2$ can be achieved by some  $(M, g,e^{-f}dv)$, for instance:
\begin{example} 
Gaussian shrinking soliton $(\mathbb{R}^n,g_{can},\frac{|x|^2}{4})$.   $f=\frac{|x|^2}{4}$, $x\in \mathbb{R}^{n}$.  $\ric_f=\frac 12g_{can}$, $\lambda_1(\Delta_f)=\frac12$ with multiplicity $n$.
\end{example}
\begin{example} Cylinder shrinking solitons:

\noindent  
$\mathbb{S}^{n-k}(\sqrt{2(n-k-1)})\times\mathbb{R}^k$, $n-k\geq 2, k\geq 1$ with  the product metric  $g$ and $f=\frac{|t|^2}{4}, t\in \rr^k$. Here  $\mathbb{S}^{n-k}(\sqrt{2(n-k-1)}) $ is the $(n-k)$-dimensional round sphere of radius $\sqrt{2(n-k-1)} $. By a direct calculation,  $\ric_f=\frac 12g$, $\lambda_1(\Delta_f)=\frac12$ with multiplicity $k$.
\end{example}
In Section \ref{sec-eigen}, we  prove that
\begin{theorem}\label{eigen}Let $(M^n, g,e^{-f}dv)$ be a complete smooth metric measure space with $\ric_f\ge \frac a2g$ for some positive constant $a$. Then  the equality in \eqref{ine-lambda-1} holds and $\lambda_1(\Delta_f)=\frac a2$ with multiplicity $k\geq 1$ if and only if 
\begin{itemize}
\item $1\leq k\leq n$;
 \item $M$ is a noncompact manifold which is isometric to $\Sigma^{n-k}\times \rr^k$ with the product metric for some complete $(n-k)$-dimensional manifold
$(\Sigma, g_{\Sigma})$ satisfying that $\ric_f^\Sigma\ge\frac a2g_{\Sigma}$ and $\lambda_1( \Delta_f^{\Sigma})>\frac a2$;
 \item By passing an isometry, for $(x, t)\in \Sigma^{n-k}\times\rr^{k}$,
 $$f(x,t)=f(x,0)+\frac{a}{4}|t|^2.$$
 In the above 
$\ric_f^\Sigma$ and  $\lambda_1( \Delta_f^{\Sigma}) $ denote the Bakry-\'Emery Ricci curvature  of $\Sigma$  and the first nonzero eigenvalue  of the drifted Laplacian $\Delta_f$ on  $\Sigma$ respectively, where  the restriction of $f$ on $\Sigma$, still denoted by $f$, is defined as $f(x,0)$.
\end{itemize}

\end{theorem}
Theorem  \ref{eigen} says that if the lower bound $\frac a2$ of the first nonzero eigenvalue $\lambda_1(\Delta_f)$ is achieved with multiplicity $k\geq 1$, then  $(M^n, g,e^{-f}dv)$, up to an isometry, must split off  a gradient   shrinking Ricci soliton $(\mathbb{R}^{k}, g_{can}, \frac{a}{4}|t|^2)$, $t\in \mathbb{R}^{k}$, $1\leq k\leq n$,   which is a Gaussian shrinking soliton by a scaling of function $f$.
The method to  Theorem \ref{eigen} is that:   the discreteness of the spectrum of $\Delta_f$ 
implies  the existence of solutions of eigenvalue problem of $\Delta_f$. Thus we may  use Bochner formula,   a technique taken to prove the Lichnerowicz-Obata theorem.  This way  also gives a different proof of the lower bound $\frac a2$ of $\lambda_1(\Delta_f)$ in Theorem \ref{th-eigen}.  Theorem \ref{eigen} can be applied to complete gradient shrinking Ricci solitons, namely,  $(M, g, f)$ with $\ric_f=\frac a2g$, where  $a$ is a positive constant. Therefore  we prove a splitting theorem for complete gradient shrinking Ricci solitons.

Later in Section \ref{sec-soliton}  we further discuss the case of  complete gradient shrinking Ricci solitons. We obtain an upper bound estimate of $\lambda_1(\Delta_f)$  for noncompact gradient shrinking Ricci solitons (Theorem \ref{th-soliton}). 

We would like to mention that  for closed Riemannian manifolds with $\text{Ric}_f\geq \frac a2g$  for some constant $a>0$,  Andrews-Ni \cite{BN}, and Futaki-Li-Li \cite{FLL} obtained  lower bound estimates for $\lambda_1(\Delta_f)$, which depend on the diameter of the manifolds. The authors of \cite{BN} stated, by constructing examples,  that their estimates are sharp when Bakry-\'Emery  Ricci curvature is not constant. See  details in \cite{FS}, \cite{BN} and \cite{FLL}.  On the other hand, the spectrum properties of $\Delta_f$ and the Laplacian $\Delta$ for complete smooth metric measure space $(M,g, e^{-f}dv)$ with various hypotheses on $\ric_f$ have been studied (cf \cite{CL},  \cite{Leo}, \cite{LZ}, \cite{MW1}, \cite{MW2}, \cite{MW3} and the references therein).  One interesting fact is that  contrary  to the drifted laplacian, the essential  spectrum of the Laplacian $\Delta$ for noncompact  gradient shrinking Ricci solitons is $[0,+\infty)$ (\cite{LZ}, \cite{CL}).

Secondly in  this paper, we study the case of complete self-shrinkers of  the mean curvature flows in  Euclidean space, whose  Bakry-\'Emery Ricci curvature  in general  may not be  bounded below by a positive constant. For a self-shrinker, the drifted Laplacian $\Delta_f$ with $f=\frac{|x|^2}{4}$
 is just the operator $\mathcal{L}=\Delta-\frac12\langle x, \nabla\cdot\rangle$ (see \cite{CM}),  where $x$ denotes the position vector in the ambient Euclidean space.  To study the discreteness of spectrum of $\mathcal{L}$, we need  Ecker's  logarithmic Sobolev inequality \cite{E} for self-shrinkers.  On the other hand, for a self-shrinker, Ding-Xin  \cite{DX} proved that its proper immersion implies Euclidean volume growth and hence finite weighted volume. In \cite{CZ}, we proved that finite weighted volume implies proper immersion. Hence for a self-shrinker, the proper immersion, Euclidean volume growth,  polynomial volume growth and finite weighted volume are equivalent each other.  Using Proposition \ref{sobo-prop-2}  together with the above facts, we  prove that  the spectrum of $\mathcal{L}$ is discrete for a properly immersed self-shrinker.  Moreover, the discreteness of spectrum let us obtain a logarithmic Sobolev inequality different from Ecker's. More precisely, we prove that 
\begin{theorem} \label{th-self-logsob}Let $M^n$ be a properly immersed complete $n$-dimensional self-shrinkers in the Euclidean space $\mathbb{R}^{n+p}$, $p\geq 1$. Then

1) the spectrum of $\mathcal{L}$ is discrete and consequently the first nonzero eigenvalue $\lambda_1$ of $\mathcal{L}$ has finite multiplicity and satisfies $\lambda_1\leq \frac12$;

2)  $M$ satisfies a  logarithmic Sobolev inequality with constants $(C',0)$:  for $u\in C_0^{\infty}(M)$,
\begin{align}\label{ine-ls-13}
&\int_{M}u^2\log  u^2e^{-\frac{|x|^2}{4}}dv-\left(\int_{M}u^2e^{-\frac{|x|^2}{4}}dv\right)\log\frac{\int_{M}u^2e^{-\frac{|x|^2}{4}}dv}{\int_Me^{-\frac{|x|^2}{4}}dv}\nonumber\\
&\qquad\qquad\qquad \leq C'\int_{M}|\nabla u|^2e^{-\frac{|x|^2}{4}}dv,
\end{align}
where 
\[ C'=4+\frac{3c(n)+n+1+\log\int_Me^{-\frac{|x|^2}{4}}dv}{\lambda_1}.\]
In the above, $c(n)$ denotes the constant in Ecker's logarithmic Sobolev inequality \eqref{ine-ls-11}, which depends only on $n$.
\end{theorem}
\begin{remark}\label{rema-entropy}
 We would like to mention that  $\mathbb{R}^n$ has the least entropy of all self-shrinkers and hence the least weighted volume $\int_{\mathbb{R}^n}e^{-\frac{|x|^2}{4}}dv=(4\pi)^{\frac n2}$ (see its proof in Section \ref{sec-self}). Hence  $\log \int_Me^{-\frac{|x|^2}{4}}dv>0$ and  $C'>0$.
\end{remark}

 In Section \ref{sec-self}, we also give a condition on the second fundamental form under which  self-shrinkers are properly immersed and obtain an application  of Theorem   \ref{th-eigen} (see Theorem \ref{th-self-1}). In particular, if the principal curvatures $\eta_i, i=1,\ldots, n,$  of a self-shrinker hypersurface $M^n$ in $\mathbb{R}^{n+1}$ satisfy $\displaystyle\sup_{1\leq i\leq n}\eta_i^2\leq\delta<\frac 12$  for some nonnegative constant $\delta$, then it is properly immersed (see Corollary \ref{cor-self}).

The rest of this paper is organized as follows:  In Section \ref{sec-embed}, as a preparation,  we give the proof of  Theorems \ref{th-eigen} and  \ref{discrete};  In Section \ref{sec-eigen} we prove Theorem \ref{eigen}; In Section \ref{sec-soliton} we prove Theorem \ref{th-soliton};  In Section 5 we prove  Theorems \ref{th-self-logsob} and \ref{th-self-1}; In Appendix, we prove Proposition \ref{sobo-prop-2} for the sake of completeness of proof.
\bigskip

\noindent {\bf Acknowledgment.} This work was done while the authors were visiting the Department of Mathematics, MIT.  The authors would like to thank Tobias Colding and William Minicozzi  II for invitation and interest in this work. The  authors are grateful to  Aaron Naber for giving them  the idea to prove Theorem \ref{discrete} and useful suggestions. Finally, the authors acknowledge the Brazilian's {\it Ci\^encias sem fronteiras} Program.

 \section{ Compact embedding of weighted Sobolev space and discreteness of spectrum } \label{sec-embed}
 
 In this section, we give notations and some known results as preparation.
 Let  the triple $(M^n,g,e^{-f}dv)$ be an $n$-dimensional  complete smooth metric measure space.   For convenience, here and thereafter, unless otherwise
specified, we denote by $\mu$ the measure induced by the weighted volume element $e^{-f}dv$, i.e.,  $d\mu:= e^{-f}dv$.  
 Let  $L^2(M, \mu)$  denote  the space of square-integrable functions on $M$ with respect to the measure $\mu$.  Let $H^1(M)$ denote the space of functions in $L^2(M, \mu)$ whose gradient is square-integral with respect to the measure $\mu$. $L^2(M, \mu)$ and $H^1(M, \mu)$ are Hilbert spaces with the following norms respectively:
\begin{equation*}
\begin{split}
\|u\|_{L^2(M,\mu)}:&=\left(\int_Mu^2d\mu\right)^{\frac12},\\
\|u\|_{H^1(M,\mu)}:&=\left(\int_M(u^2+|\nabla u|^2)d\mu\right)^{\frac12}.
\end{split}
\end{equation*}
When $M$ is complete, Sobolev space  $H^1(M, \mu)$ is  
  the closure of the space $C_0^{\infty}(M)$ of all compactly-supported  smooth functions on $M$, with respect to  the norm  in $H^1(M, \mu)$.  
  
 For $(M^n,g,e^{-f}dv)$, the spectrum gap of $\Delta_f$ is the following infimum: 
$$\displaystyle\inf\left\{\frac{\int_M|\nabla u|^2d\mu}{\int_Mu^2d\mu};  u\in H^1(M, \mu), u\not\equiv 0, \int_Mud\mu=0\right\}.$$
 
  $(M^n,g,e^{-f}dv)$ is said to satisfy a Poincar\'e inequality with constant $C$ if  there exists  some positive constant $C$ so that 
 \begin{equation} \label{ine-poincare}\int_{M}|\nabla u|^2d\mu\geq C\int_{M}u^2d\mu, 
\end{equation}
for all $ u\in H^1(M, \mu)$ with $ \int_Mud\mu=0$.

Clearly, if   $(M^n,g,e^{-f}dv)$ satisfies Poincar\'e inequality  \eqref{ine-poincare}, then the spectrum gap of $\Delta_f$ is not less than $C$.

\begin{definition}  Let  $(M^n,g,e^{-f}dv)$ be a complete smooth metric measure space   with $\mu(M)=\int_Me^{-f}dv<\infty$. $(M^n,g,e^{-f}dv)$ is said to satisfy 
 a logarithmic Sobolev  inequality  with constants $(C,D)$ if there exist some constants  $C\in (0,\infty), D\in [0,\infty)$ so that   
   \begin{align}\label{ine-ls-15}
\int_Mu^2\log (u^2)d\mu-&\left(\int_Mu^2d\mu\right)\log\left(\frac {\int_Mu^2d\mu}{\mu(M)}\right)\nonumber\\
&\leq 2C\int_M|\nabla u|^2d\mu+2D\int_Mu^2d\mu,
\end{align}
 for all $u\in H^1(M, \mu)$.
 
   When $D\neq 0$, \eqref{ine-ls-15} is called a defective  logarithmic Sobolev  inequality. 
\end{definition}

\eqref{ine-ls-15} is equivalent to that for all $u\in H^1(M, \mu)$ with $\int_Mu^2d\mu=\mu(M)$,
    \begin{align}\label{ine-ls-16}
\int_Mu^2\log (u^2)d\mu\leq 2C\int_M|\nabla u|^2d\mu+2D\int_Mu^2d\mu.
\end{align}
  
If $D=0$, \eqref{ine-ls-15} and  \eqref{ine-ls-16}  turn to the following inequalities respectively:
  \begin{align}\label{ine-ls-15-1}
\int_Mu^2\log (u^2)d\mu-&\left(\int_Mu^2d\mu\right)\log\left(\frac {\int_Mu^2d\mu}{\mu(M)}\right)\leq 2C\int_M|\nabla u|^2d\mu,
\end{align}
for all $u\in H^1(M, \mu)$; and
\begin{equation}\label{varphi}
\int_M u^2\log u^2d\mu\le 2C\int_M|\nabla u|^2d\mu,
\end{equation}
 for all $u\in H^1(M, \mu)$ satisfying $\int_Mu^2d\mu=\mu(M)$.
 
  $(M^n,g,e^{-f}dv)$ satisfies a logarithmic Sobolev  inequality \eqref{ine-ls-15}  if  and only if it holds for $u\in C_0^{\infty}(M)$.  The same conclusion holds for \eqref{ine-ls-16}, \eqref{ine-ls-15-1} and \eqref{varphi} respectively (cf Remark \ref{rema-1} in Appendix). 
  It is known that the following Proposition \ref{sobo-prop-2} holds. Since we couldn't find a suitable reference, for the sake of completeness of proof, we include its proof in Appendix.

\begin{prop}\label{sobo-prop-2} Let $(M, g,e^{-f}dv)$ be a complete smooth metric measure space with finite weighted volume  $\mu(M)=\int_Me^{-f}dv$. If   the  logarithmic Sobolev inequality \eqref{ine-ls-15} with constants $(C,D)$ holds on  $(M,g,e^{-f}dv) $, 
then the inclusion  $H^1(M, \mu)\subset L^2(M, \mu)$ is  compactly embedded and equivalently the spectrum of $\Delta_f$ for $M$ is discrete.
\end{prop}

Now suppose that  $(M, g,e^{-f}dv)$ has $\ric_f\ge \frac a2g$ for some constant  $a>0$.   Recently it was obtained by Morgan  \cite{M}  (cf \cite{WW}) that $M$ has finite weighted volume $\mu(M)$. Hence with the logarithmic Sobolev inequality \eqref{varphi-1.1}  by Bakry-\'Emery and finiteness of the weighted volume $\mu(M)$ by Morgan, 
Proposition \ref{sobo-prop-2} implies 
  the following result obtained by Hein-Naber \cite{HN}:

\begin{theorem}\label{discrete}({Hein-Naber}) Let $(M, g,e^{-f}dv)$ be a complete smooth metric measure space with $\ric_f\ge \frac a2g$ for some positive  constant $a$. Then the inclusion $H^1(M, \mu)\subset L^2(M, \mu)$ is  compactly embedded and equivalently the spectrum of $\Delta_f$ for $M$ is discrete.
\end{theorem}

From Theorem \ref{discrete},  the spectrum of $\Delta_f$ is just the set of points which are both eigenvalues of $\Delta_f$  with finite multiplicity and the isolated points of the spectrum (cf \cite{G} Definition 10.1).  Since the weighted volume $\mu(M)$  is finite,   $0$ is the least eigenvalue with multiplicity one and nonzero constant functions are  the associated eigenfunctions. Thus the set of all eigenvalues of $\Delta_f$, counted with multiplicity,  is an increasing sequence $$0=\lambda_0(\Delta_f)< \lambda_1(\Delta_f)\leq\cdots$$ with $\lambda_i(\Delta_f)\to \infty$ as $i\to \infty$.  The variational characterization of $\lambda_i(\Delta_f)$ states that $\lambda_1(\Delta_f)$ is the spectrum gap of $\Delta_f$.
Moreover, there exists a countable orthonormal base $\{\psi_i\}$ of $L^2(M, \mu)$ so that each $\psi_i$ is an eigenvector of $\Delta_f$ associated with the eigenvalue $\lambda_i(\Delta_f)$.
Now we give the proof of Theorem \ref{th-eigen}. 
 
\noindent {\it Proof of Theorem \ref{th-eigen}.} From Theorem \ref{discrete}, the spectrum of $\Delta_f$ is discrete.  From the statements following Theorem \ref{discrete},  the first nonzero eigenvalue $\lambda_1(\Delta_f)$ is the spectrum gap of $\Delta_f$.
 We estimate  the lower bound of $\lambda_1(\Delta_f)$ as follows:  under the hypothesis of theorem, $M$ satisfies  that $\mu(M)<\infty$ and 
the  logarithmic Sobolev inequality \eqref{varphi-1.1} with constants $(\frac{4}{a},0)$. It is well known that  a  logarithmic Sobolev inequality \eqref{varphi} with constant $(C,0)$  implies a Poincar\'e inequality \eqref{ine-poincare} with constant $\frac{1}{C}$ (cf \cite{Le} Prop 2.1 or \cite{GZ} Th 4.9).
Hence the spectrum gap $\lambda_1(\Delta_f)\geq\frac a2$.   

\qed

\section{First nonzero eigenvalue estimate}\label{sec-eigen}

In this section we prove Theorem \ref{eigen}.  The proof also give another proof of Theorem \ref{th-eigen} without using Poincar\'e inequality.
\medskip

\noindent {\it Proof of Theorem \ref{eigen}}.  It suffices to consider, by a scaling of the metric $g$,  the case of $(M,g, e^{-f})$ with $\ric_f\geq \frac{1}{2}g$ (the conclusion for general case can be obtained by the relation between the original and scaled metrics).
 The discreteness of spectrum of $\Delta_f$ guarantees the existence of solutions of eigenvalue problem of $\Delta_f$.
Assume that $u$ is a nonconstant eigenfunction  of $\Delta_f$ corresponding to an eigenvalue $\lambda$, i.e., 
\begin{equation}\label{eq-eigen-1}\Delta_f u+\lambda u=0, \quad \int_Mu^2d\mu<\infty.
\end{equation}
It is known that $u\in H^1(M,\mu)\cap C^{\infty}(M)$. By $\ric_f\geq \frac{1}{2}g$,  \eqref{eq-eigen-1} and the weighted Bochner formula:
\[
\frac 12\Delta_f |\nabla u|^2=|\nabla^2 u|^2+\langle\nabla u, \nabla(\Delta_fu)\rangle+\ric_f(\nabla u,\nabla u\rangle,
\]
we have 
\begin{equation}\label{delta-f}
\frac12\Delta_f |\nabla u|^2\ge|\nabla^2 u|^2+(\frac 12-\lambda)|\nabla u|^2.
\end{equation}
If $M$ is compact, integrating \eqref{delta-f}, we have
\begin{equation}\label{compact}
\int_M|\nabla^2 u|^2d\mu+\int_M(\frac 12-\lambda)|\nabla u|^2d\mu\leq 0.
\end{equation}
If $M$ is noncompact,  we claim that \eqref{compact} also holds. Given a fixed point $p\in M$, let $B_R$ denote the geodesic sphere of $M$ of radius $R$ centered at $p$.  Let $\phi$  be the nonnegative cut-off function satisfying  that  $\phi$ is $1$ on $B_R$,  $|\nabla\phi|\leq 1$ on $B_{R+1}\setminus B_R$, and $\phi=0$ on $\Sigma\setminus B_{R+1}$. Multiplying \eqref{delta-f} by $\phi^2$ and then
integrating, we get 
\begin{equation}\label{ine-1}
\frac12\int_M\phi^2\Delta_f |\nabla u|^2d\mu\ge\int_M\phi^2|\nabla^2 u|^2d\mu+(\frac 12-\lambda)\int_M\phi^2|\nabla u|^2d\mu.
\end{equation}
On the other hand, by the weighted Green formula,
\begin{align}
\int_M\phi^2\Delta_f |\nabla u|^2d\mu&= -\int_M\langle\nabla\phi^2,\nabla|\nabla u|^2\rangle d\mu\nonumber\\
&= 4\int_M\phi\langle\nabla_{\nabla\phi}\nabla u,\nabla u\rangle d\mu\label{eq-1}\\
&= 4\int_M\phi (\nabla^2u)(\nabla\phi,\nabla u)d\mu.\nonumber
\end{align}
By $2ab\leq \epsilon a^2+\frac{b^2}{\epsilon}$,  for any $\epsilon>0$, 
\begin{align}
2\phi (\nabla^2u)(\nabla\phi,\nabla u)&= 2\displaystyle\sum_{i,j=1}^n\phi(\nabla^2u)_{ij}\phi_iu_j\nonumber\\
&\leq  \displaystyle\sum_{i,j=1}^n\big[\epsilon\phi^2(\nabla^2u)_{ij}^2+\frac{1}{\epsilon}\phi_i^2u_j^2\big]\label{ine-hession}\\
&=  \epsilon \phi^2|\nabla^2u|^2+\frac{1}{\epsilon}|\nabla\phi|^2|\nabla u|^2.\nonumber
\end{align}
In the above, the subscripts $i, j$ denote the covariant derivatives with respect to $e_i,e_j$ respectively, where $\{e_i\}$ denotes a local orthonormal frame on $M$. Substituting  \eqref{ine-hession} into \eqref{eq-1}, we have
\begin{equation}\label{delta-2}
\int_M\phi^2\Delta_f |\nabla u|^2d\mu\leq 2\epsilon \int_M\phi^2|\nabla^2u|^2d\mu+\frac{2}{\epsilon}\int_M|\nabla\phi|^2|\nabla u|^2d\mu.
\end{equation}
Combining \eqref{ine-1} and \eqref{delta-2}, it holds that
\begin{align*}
(1-\epsilon)\int_M\phi^2|\nabla^2u|^2d\mu&\leq  \frac{1}{\epsilon}\int_M|\nabla\phi|^2|\nabla u|^2d\mu+(\lambda-\frac 12)\int_M\phi^2|\nabla u|^2d\mu.
\end{align*}
Noting  $\int_M|\nabla u|^2<\infty$ and  letting $R\to \infty$ in the above inequality, by the monotone convergence theorem,  we have $\int_M|\nabla^2u|^2d\mu<\infty$. Furthermore,   observe that
\begin{align}\left |2\phi (\nabla^2u)(\nabla\phi,\nabla u)\right |&=2\left |\displaystyle\sum_{i,j=1}^n\phi(\nabla^2u)_{ij}\phi_iu_j\right|\nonumber\\
&\leq \epsilon \displaystyle\sum_{i,j=1}^n\phi^2(\nabla^2u)_{ij}^2|\phi_i|+\frac{1}{\epsilon}\displaystyle\sum_{i,j=1}|\phi_i|u_j^2\label{ine-est-1}\\
&\leq \epsilon \sqrt{n}\phi^2|\nabla^2u|^2|\nabla \phi|+\frac{\sqrt{n}}{\epsilon}|\nabla\phi||\nabla u|^2.\nonumber
\end{align}
So \eqref{ine-est-1} implies that
\begin{align}\left|\int_M\phi^2\Delta_f |\nabla u|^2d\mu\right |
&=\left |4\int_M\phi (\nabla^2u)(\nabla\phi,\nabla u)d\mu\right |\nonumber\\
&\leq 2\epsilon \sqrt{n}\int_M\phi^2|\nabla^2u|^2|\nabla \phi |d\mu+\frac{2\sqrt{n}}{\epsilon}\int_M|\nabla\phi||\nabla u|^2d\mu\label{ine-est-2}\\
&\leq 2\epsilon \sqrt{n}\int_{B_{R+1}\setminus B_R}|\nabla^2u|^2d\mu+\frac{2\sqrt{n}}{\epsilon}\int_{B_{R+1}\setminus B_R}|\nabla u|^2d\mu.\nonumber
\end{align}
Letting $R\to\infty$ in \eqref{ine-est-2}, the right-hand side convergences to zero.  Thus $$\displaystyle\lim_{R\to\infty}\int_M\phi^2\Delta_f |\nabla u|^2d\mu\to 0.$$
Letting $R\to \infty$ in \eqref{ine-1} and using the dominate convergence theorem, we have \eqref{compact} holds:
\begin{equation*}\label{noncompact}
\int_M|\nabla^2 u|^2d\mu+(\frac 12-\lambda)\int_M|\nabla u|^2d\mu\le 0.
\end{equation*}
So the claim holds.
Since $u$ is not constant, from \eqref{compact},   $ \lambda\ge \frac 12$. This implies that $ \lambda_1(\Delta_f)\ge \frac 12$, as in Theorem \ref{th-eigen}.

Now we consider the case of the equality. From the proof,    $ \lambda= \frac 12$ if and only if 
\begin{align}
\label{eq1}
\nabla^2 u&=0,\\
\ric_f(\nabla u,\nabla u)&=\frac 12|\nabla u|^2\\
\label{eq1-1}
\Delta_fu+\frac 12 u&=0, \quad \int_Mu^2d\mu<\infty.
\end{align}
By \eqref{eq1},  $\Delta u=0$ and hence by \eqref{eq1-1}, 
\begin{equation}\label{eq2}
-\langle\nabla f, \nabla u\rangle+\frac 12 u=0.
\end{equation}
Thus $u$ is a nonconstant harmonic function and $M$ must be noncompact. Moreover,  \eqref{eq1} together with the fact $u$ is not constant means that $\nabla u$ is a nontrivial parallel vector field and hence  implies that $M$ must be isometric to a product manifold $\Sigma^{n-1}\times \rr$ for some complete manifold $(\Sigma, g_{\Sigma})$. Besides  $u$ is constant on  the level set $\Sigma\times \{t\}, t\in \rr$.  Without lost of generality, suppose that  $\Sigma:=u^{-1}(\{0\})$.  By passing an isometry, we may assume that $M= \Sigma^{n-1}\times \mathbb{R}$. Take $(x,t)\in \Sigma^{n-1}\times\mathbb{R}$.   From (\ref{eq2}), $\frac{\partial f}{\partial t}=\frac{t}{2}$. So  
\begin{equation}\label{eq-f}
f(x, t)=\frac {t^2}{4}+f(x,0).
\end{equation}
From \eqref{eq-f}, for any vector field $X\in T\Sigma$ and the unit normal $\nu$ to $\Sigma$, it holds that
$$\nabla^2f(X, \nu)=0, \nabla^2f(\nu, \nu)=\frac12,$$
$$\nabla^2f(X, X)=(\nabla^{\Sigma})^2f(X,X).
$$
 \noindent Here and thereafter we denote still by $f$ the restriction of $f$ on the corresponding submanifolds, for instance, $f|_{\Sigma}$ by $f$.
Also the superscripts $\Sigma$, $\rr$ denote the corresponding geometric quantities of $\Sigma$ and $\rr$ respectively, for instance, $\nabla^{\Sigma}$ denotes the connection of $\Sigma$.
 By a direct computation,  the Ricci curvature satisfies that
 on $\Sigma$ $$\ric(X,X)=\ric^{\Sigma}(X,X), \quad \ric(\nu,\nu)=\ric(\nu,X)=0.$$ Therefore
\[\ric_f^\Sigma\ge \frac12g_{\Sigma}.
\]
By Theorem \ref{discrete}, the spectrum of $\Delta_f^{\Sigma}$ on $\Sigma$ for $L^2_f(\Sigma)$ is also discrete. By direct computation, we have the identity:
\begin{align}
\Delta_fu(x,t)&=\Delta_f^{\Sigma}u|_{\Sigma\times\{t\}}(x)+\Delta_f^{\rr}u|_{\{x\}\times \rr}(t)\nonumber\\
&=\Delta_f^{\Sigma}u|_{\Sigma\times\{t\}}(x)
+(\frac{d^2}{ dt^2}-\frac{t}{2}\frac{d }{dt})u|_{\{x\}\times \rr}(t),\label{eq-delta}
\end{align}
where by abuse of notations, $\Delta_f^{\Sigma}$ and $\Delta_f^{\rr}$ denote the drifted Laplacians of $\Sigma\times \{t\}$ and $\{x\}\times\rr$, which act on functions $u|_{\Sigma\times\{t\}}$ and $u|_{\{x\}\times \rr}$ respectively.
By the theory of functional analysis, the discreteness of the spectrum of  $\Delta_f^{\Sigma}$ implies that there exists a complete orthonormal system for space $L^2(\Sigma, e^{-f}d\sigma)$ consisting of eigenfunctions of $\Delta_f^{\Sigma}$, where $d\sigma$ is the volume element of $\Sigma$ induced by the metric $g_{\Sigma}$ of $\Sigma$.  Also, for the operator $\frac{d^2}{dt^2}-\frac{t}{2}\frac{d}{dt}, t\in \rr$, it is known that its spectrum  on $L^2(\rr,e^{-\frac{t^2}{4}}dt)$ is discrete and the so-called Hermite polynomials are  orthonormal eigenfunctions,  which form a complete orthonormal system for space $L^2(\rr, e^{-\frac{t^2}{4}}dt)$.  By these  facts and  \eqref{eq-delta},  one can verify  that the products of   the orthonormal eigenfunctions of  $\Delta_f^{\Sigma}$ and the orthonormal eigenfunctions of $\frac{d^2}{dt^2}-\frac{t}{2}\frac{d }{dt}$ are the eigenfunctions of $\Delta_f$ and, by a standard argument,  form a complete system for  space $L^2(M, e^{-f}dv)$.  Therefore the eigenvalues  $\sigma(\Delta_f)$  of $M$ counted with multiplicity are just the sums of the corresponding eigenvalues $\sigma(\Delta_f^{\Sigma})$ of $\Sigma$ and  $\sigma(\Delta_f^{\rr})$ of  $\rr$ counted with multiplicity. It is known that 
\[\sigma(\Delta_f^{\rr})=\{0, \frac12, 1, \frac32, \cdots\},
\] 
where the first nonzero eigenvalue $\frac12$ has multiplicity one.
Hence
\begin{equation}\label{sigma-1}
\sigma(\Delta_f^{M})=\{0, \frac12, \min\{\lambda_1(\Delta_f^{\Sigma}),1\},  \cdots\},
\end{equation}
where $\lambda_1(\Delta_f^{\Sigma})$ is the first nonzero eigenvalue of $\Delta^{\Sigma}_f$. 

To conclude the proof,  we claim that {\it if the multiplicity of $\lambda_1(\Delta_f)=\frac12$ is  $k$, then $M$ is isometric to $\Sigma^{n-k}\times\rr^k$ with $\ric_f^{\Sigma^{n-k}}\ge \frac12g_{\Sigma^{n-k}}$,
$\lambda_1(\Delta_f^{\Sigma^{n-k}})> \frac12$, and
\begin{equation}\label{eq-ff}
f(x_1,\ldots,x_{n-k}, t_1,\ldots, t_k)=f(x,0)+\frac{1}{4}\displaystyle\sum_{i=1}^kt_i^2,
\end{equation}
 where $(x,t)=(x_1,\ldots,x_{n-k}, t_1,\ldots, t_k)\in \Sigma^{n-k}\times\rr^{k}$. }
 
  In the following proof, we will use  $\Sigma^j$ with superscript $j$ to distinguish different $\Sigma$, whose dimension is $j$.  We will prove the claim by induction. 
  
  First suppose that the multiplicity $k=1$.  By the proof before, we know that $M=\Sigma^{n-1}\times \rr$, $\ric_f^{\Sigma^{n-1}}\ge \frac12g_{\Sigma^{n-1}}$ and $f$ is as in \eqref{eq-f}. From \eqref{sigma-1}, we know that $\lambda_1(\Delta_f^{\Sigma^{n-1}})>\frac12$. So the claim holds for $k=1$.
  
Next suppose that the conclusion of the claim holds for multiplicity $ k-1$ and  $\lambda_1(\Delta_f)=\frac12$ has  multiplicity $k$. Then we have that $M=\Sigma^{n-1}\times \rr$ with $\ric_f^{\Sigma^{n-1}}\ge \frac 12g$ and $f=\frac{t_k^2}{4}+f|_{\Sigma^{n-1}}$, where $t_k\in\rr$. By \eqref{sigma-1}, $\lambda_1(\Delta_f^{\Sigma^{n-1}})= \frac 12$ must have  multiplicity $k-1$. Hence by hypothesis of induction, $\Sigma^{n-1}=\Sigma^{n-k}\times\rr^{k-1}$ with $\ric_f^{\Sigma^{n-k}}\ge \frac 12g_{\Sigma^{n-k}}$, $\lambda_1(\Delta_f^{\Sigma^{n-k}})>\frac12$ and 
$$f|_{\Sigma^{n-1}}(x_1,\ldots,x_{n-k}, t_1,\ldots, t_{k-1})=f(x,0)+\frac{1}{4}\displaystyle\sum_{i=1}^{k-1}t_i^2,$$
 where $(x_1,\ldots,x_{n-k}, t_1,\ldots, t_{k-1})\in \Sigma^{n-k}\times\rr^{k-1}$. Thus $M=\Sigma^{n-k}\times\rr^{k}$ and $f$ is as \eqref{eq-ff}. So the conclusion of the claim holds for $k$. By induction,   the claim is proved. Therefore we prove the conclusion of theorem under the assumption that $\lambda_1(\Delta_f)$ has multiplicity $k\geq 1$.
 
  The inverse is a direct computation.
 Therefore  we complete the proof of theorem.
 
\qed

Theorem \ref{eigen} has the following corollaries.

\begin{cor}\label{cor-compact}Let $(M^n, g,e^{-f}dv)$ be a closed smooth metric measure space with $\ric_f\ge \frac a2g$, where constant $a$ is positive, then  the  first nonzero eigenvalue $\lambda_1(\Delta_f)$ of $\Delta_f$ satisfies $$\lambda_1(\Delta_f)>\frac a2.$$
\end{cor}

\begin{cor}\label{cor3}Let $(M^n, g,e^{-f}dv)$ be a complete smooth metric measure space with $\ric_f\ge \frac a2g$, where constant $a$ is positive.  If 
$ \frac a2$ is the  first nonzero eigenvalue $\lambda_1(\Delta_f)$ with multiplicity  $ n-1$,
then $M$ is isometric to the Euclidean space $ \rr^n$ and $f$ can be expressed as 
\begin{equation}\label{eq-f-cor}f(x_1, \cdots, x_n)=\varphi(x_1)+\frac{a(x_2^2+\cdots +x_{n}^2)}4,
\end{equation}
where $\varphi$ is smooth function satisfying $\varphi''\ge \frac a2$. 

\end{cor}
\begin{proof}
From Theorem \ref{eigen}, $M$ is isometric to $\Sigma\times \mathbb{R}^{n-1}$. $\Sigma $  has dimension $1$. Meanwhile it is known by \cite{M}  that  the fundamental group $\pi_1(M)$ is finite. So   $\Sigma$ must be $\mathbb{R}$,  not a circle and  $M$ is isometric to $\mathbb{R}^n$. In this case,  $\ric_f=\nabla^2f$. Using the general expression of $f$ in Theorem \ref{eigen},  we obtain \eqref{eq-f-cor}   by direct computation.

\end{proof}

With Theorem \ref{eigen}, we may further estimate other eigenvalues:

\begin{cor}Let $(M, g,e^{-f}dv)$ be a complete smooth metric measure space with $\ric_f\ge \frac 12g$. Suppose that  the first nonzero eigenvalue $\lambda_1(\Delta_f)=\frac12$ with multiplicity $k$.  If  the splitting $M=\Sigma^{n-k}\times \rr^k$, $k\geq 1$,  satisfies that   $\Sigma $ is compact and $f$ is constant on $\Sigma$, then the next eigenvalue $\lambda_2(\Delta_f)$ of $\Delta_f$ on $M$ satisfies $ \lambda_2(\Delta_f)\ge \frac 12\frac {n-k}{n-k-1}.$ Moreover the equality holds if and only if $\Sigma$ is isometric to the round sphere $\mathbb{S}^{n-k}$ of radius $\sqrt{2(n-k-1)}$.
\end{cor}

\begin{proof} Since  $f$ is constant on $\Sigma$, $\ric^\Sigma=\ric_f^\Sigma\ge \frac12g$ on $\Sigma$ and $n-k\ge 2$. Analogous to the proof of Theorem \ref{eigen}, the eigenvalues of $\sigma(\Delta_f)$  of $M$ are the sums of the corresponding eigenvalues  $\sigma(\Delta_f^\Sigma)$ of $\Sigma$ and  $\sigma(\Delta_f^{\rr^k})$ of $\rr^k$ with restricted $f$ on $\Sigma$ and $\mathbb{R}^k$ respectively. By  Theorem \ref{eigen}, we know that 
$$f=\displaystyle\sum_{i=1}^k\frac{t^2_i}{4}+f|_{\Sigma},$$
where $(t_{1},\ldots, t_k)\in \rr^{k}$.
In this case
$$\sigma(\Delta_f^{\rr^k})=\{0,\underbrace{\hbox{$\frac{1}{2},\ldots, \frac12$}}_{\hbox{k}},1,  \cdots\}.$$
\begin{equation}\label{set-eigen}\sigma(\Delta_f^{M})=\{0, \underbrace{\hbox{$\frac{1}{2},\ldots, \frac12$}}_{\hbox{k}}, \min\{\lambda_1(\Delta_f^{\Sigma}),1\},  \cdots\},
\end{equation}
where $\frac12$ has multiplicity $k$.
On the other hand, Note that  $\ric^\Sigma\ge \frac12g$ and  $f$ is constant on $\Sigma$. By Lichnerowicz theorem,  $$\lambda_1(\Delta_f^{\Sigma})=\lambda_1(\Delta^{\Sigma})\geq \frac{n-k}{2(n-k-1)}.$$ Observe that $\frac{n-k}{2(n-k-1)}\le 1$.  
\eqref{set-eigen} implies that $ \lambda_2(\Delta_f)\ge \frac 12\frac {n-k}{n-k-1},$
and by Obata theorem, the equality holds if and only if $\Sigma^{n-k}$ is isometric to the round sphere $\mathbb{S}^{n-k}$ of radius $\sqrt{2(n-k-1)}$.

\end{proof}

\section{Gradient shrinking Ricci soliton case}\label{sec-soliton}

Let $(M^n,g)$ be a Riemannian manifold and $f$ is a smooth function on $M$. The quadruple $(M,g,f,\rho)$
is called a {\it gradient shrinking Ricci soliton} if 
$$\ric_f=\rho g,$$
where  constant $\rho>0$. Theorem \ref{eigen} can be applied to obtain a splitting theorem for complete gradient shrinking Ricci solitons $\ric_f=\rho g$ when the lower bound $\rho$ of the first nonzero eigenvalue $\lambda_1(\Delta_f)$ can be achieved. In this section,  we will give an upper bound estimate of $\lambda_1(\Delta_f)$.

By a
scaling  of metric $g$ one can normalize $\rho=\frac{1}{2}$ so that
\begin{equation}\label{soliton-1}\ric_f=\frac{1}{2} g.
\end{equation}

It is known that \eqref{soliton-1} implies the following  identities about complete gradient
shrinking solitons.
\begin{equation}\label{sol-2}R+\Delta f=\frac n2,
\end{equation}
\begin{equation}\label{sol-3}R+|\nabla f|^2-f=C_0 
\end{equation} for some constant $C_0$. Here $R$
denotes the scalar curvature of $(M, g)$. 
We prove the following

\begin{theorem}\label{th-soliton} Let the quadruple $(M,g,f,\frac 12)$
be a complete noncompact gradient shrinking Ricci soliton. Then $1$ is an eigenvalue of $\Delta_f$ and a translation of  $f$  with some constant is an associated  eigenfunction. Moreover the first nonzero eigenvalue $\lambda_1(\Delta_f)$ of $\Delta_f$ for $M$ satisfies $\frac 12\le\lambda_1(\Delta_f)\le 1.$
\end{theorem}
\begin{proof}
Without lost of generality, by adding the constant $C_0-\frac n2$ to $f$, by \eqref{sol-3}, we can assume that  $f$ satisfies
\begin{equation}\label{sol-4} R+|\nabla f|^2-f=\frac n2. 
\end{equation}
 Then  \eqref{sol-2} and \eqref{sol-4} imply
\begin{equation}\Delta_f f+f=0.
\end{equation}
 From \cite{MR2732975}, we know that for a fixed point $p\in M$ there exists a constant $C$ such that
\[\lim_{x\to +\infty}\frac{f(x)}{r^2(x)}=\frac 14.  
\]
and the volume $\textrm{vol}(B_p( r))\le Cr^n$, where $r(x)$ is the distance function from $p$ and $B_p(r)$ is the geodesic ball of radius $r$  centered at $p$.
These facts together with   $R\ge 0$ implies that  $f\in H^1(M, \mu)$. Thus $f$ is an eigenfunction associated the eigenvalue $1$. By Theorem \ref{th-eigen}, we complete the proof.

\end{proof}

\section{ Self-shrinkers in  Euclidean space }\label{sec-self}

In this section, we discuss the discreteness of spectrum of operator $\mathcal{L}$ and logarithmic Sobolev inequality for properly immersed self-shrinkers in the Euclidean space $\mathbb{R}^{n+p}$, $p\geq 1$. First we give some notations. Let $M^n$ be an $n$-dimensional submanifold immersed in $\mathbb{R}^{n+p}$. 
The second fundamental form $B$ of $M$ is defined by $$B(X,Y)=(\overline{\nabla}_XY)^{\perp}, \quad X, Y\in T_q M,\quad q\in M,$$
where  $\perp$ denotes the projection to the normal bundle of $M$. 
The mean curvature vector ${\bf H}$ of $\Sigma$ is defined by ${\bf H}=\text{tr}B=\displaystyle\sum_{i=1}^{n}(\overline{\nabla}_{e_{i}}e_{i})^{\bot}$, where $\{e_i\}, i=1,\ldots,n$, is a local  orthonormal frame on $M$.

$M$ is called a self-shrinker 
 if its  mean curvature vector ${\bf H}$ satisfies
\begin{equation}
{\bf H}=-\frac{x^{\bot}}{2},
\end{equation}
where $x$ denotes the position vector  in $\mathbb{R}^{n+p}$. 

Take $f=\frac{|x|^2}{4}$. The restriction of $f$ on $M$, still denoted by $f$, induces a measure $\mu$ satisfying $d\mu=e^{-\frac{|x|^2}{4}}dv$, where $dv$ is the volume element of $(M,g)$.
It is known  that  the self-shrinker $M$ is an $f$-minimal submanifold in $\rr^{n+p}$ (cf \cite{CMZ}).

Define the operator 
$\mathcal{L}:=\Delta-\frac12\langle x, \nabla\cdot\rangle,$
where $\nabla$ and $Delta$ denote the gradient and Laplacian on $M$ respectively. By definition, $\mathcal{L}=\Delta_f$. The operator $\mathcal{L}$ was introduced by Colding-Minicozzi \cite{CM} to study self-shrinker hypersurfaces.  Now we prove Theorem \ref{th-self-logsob}. 
\medskip

\noindent {\it Proof of Theorem \ref{th-self-logsob}}.
For self-shrinkers in $\mathbb{R}^{n+p}$, Ecker \cite{E} proved a logarithmic Sobolev inequality  with respect to the measure $\mu$.  Precisely,  in \cite{E}, Example 2.2-(3),  take  $\rho=\frac{1}{(4\pi)^{\frac n2}}\exp(-\frac{|x|^2}{4})$ and $\lambda=2$. Then  
for all $u\in C_0^{\infty}(M)$,
\begin{align}
&   \int_{M}u^2\log  u^2d\mu-\left(\int_{M}u^2d\mu\right)\log\left(\int_{M}u^2d\mu\right)\nonumber\\
&\qquad\qquad\qquad\leq 4\int_{M}|\nabla u|^2d\mu+\left[3 c(n)+n\right]\int_{M}u^2d\mu,\label{ine-ls-11}
\end{align}
where $c(n)$ is a  constant depending only on $n$.

 Note that by \cite{CZ} and  \cite{DX}, for a self-shrinker, proper immersion, finite weighted volume and polynomial volume growth are equivalent each other (see the definition of polynomial volume growth in \cite{CM}). By hypothesis of theorem,  the weighted volume $\mu(M)=\int_{M}e^{-\frac{|x|^2}{4}}dv<\infty$. Hence  \eqref{ine-ls-11} is equivalent to that  for all $u\in C_0^{\infty}(M)$,
\begin{align}\label{ine-ls-12}
&  \int_{M}u^2\log  u^2d\mu-\left(\int_{M}u^2d\mu\right)\log\frac{\int_{M}u^2d\mu}{\mu(M)}\nonumber\\
&\qquad\qquad\qquad\leq 4\int_{M}|\nabla u|^2d\mu+\bigg[3 c(n)+n+\log\mu(M)\bigg]\int_{M}u^2d\mu.
\end{align}
Observe that  \eqref{ine-ls-12} has constants $(C, D)$, where $C=4$, $D=3c(n)+n+\log\mu(M)$. Remark \ref{rema-entropy} in Introduction states that $\mu(M)\geq (4\pi)^{\frac{n}{2}}$ (see the proof of Remark \ref{rema-entropy} later in this section) and thus  $D>0$.
 By Proposition \ref{sobo-prop-2}, $H^1(M,\mu)\subset L^2(M,\mu)$ is  compactly embedded and equivalently the spectrum of $\mathcal{L}$  is discrete.  It holds that $0$ is the least eigenvalue with multiplicity one since the nonzero constants  are associated eigenfunctions. To estimate the upper bound of the nonzero eigenvalue $\lambda_1$ of $\mathcal{L}$,
note that
\begin{align}
\Delta x_i&
=\langle \Delta x, \partial_i\rangle
=\langle { \bf H}, \partial_i\rangle=-\frac12\langle  x^{\perp}, \partial_i\rangle.\nonumber
\end{align}
\begin{align}\label{eq-eigenf-1}
\Delta_fx_i&=\Delta x_i-\langle \frac{x^{\top}}{2}, \nabla x_i\rangle
=-\frac12\langle  x^{\perp}, \partial_i\rangle-\frac12\langle x^{\top}, \partial_i\rangle
=-\frac12x_i.
\end{align}
Observe that  there exists some $i$ so that $x_i\not\equiv 0$ on $M$. Since $M$ has the polynomial volume growth,  \eqref{eq-eigenf-1} implies $x_i$ is an eigenfunction.  Hence the first nonzero eigenvalue of $\Delta_f$ satisfies $\lambda_1\leq \frac12$.

By the variational principle of eigenvalues,  the Poincar\'e inequality with constant $\lambda_1$ holds, i.e., for all $ u\in H^1(M, \mu)$ with $ \int_{M}ud\mu=0$,
\begin{align}\label{ine-spectrum}
\int_{M}|\nabla u|^2d\mu\geq \lambda_1\int_{M}u^2d\mu.
\end{align}
  It is known that a defective logarithmic Sobolev inequality with constants $(C,D)$, $D>0$ and Poincar\'e inequality  with constant $\lambda_1$ imply a logarithmic Sobolev inequality with constants $(C',0)$, $C'\leq C+\frac{D+1}{\lambda_1}$ (cf \cite{GZ} Theorem 4.9).
Hence by \eqref{ine-ls-12}, \eqref{ine-spectrum} and $\mu(M)<\infty$,  logarithmic Sobolev inequality \eqref{ine-ls-13} holds.

\qed

\noindent {\it Proof of Remark \ref{rema-entropy}.}  We prove that $\mathbb{R}^n$ has the least entropy $1$ and the least weighted volume $\mu(\mathbb{R}^n)=(4\pi)^{\frac n2}$ of all self-shrinkers in $\mathbb{R}^{n+p}$. We only need  to consider the case of  finite weighted volume, equivalently, polynomial volume growth. The following proof was given in \cite{CM} for hypersurface case, but it also holds for arbitrary codimension only with some change of notations.

  For any submanifold $M^n$ in $\mathbb{R}^{n+p}$, given $ x_0\in \mathbb{R}^{n+p}$ and $t_0 > 0$, define the functional $F_{x_0, t_0}$ of $M$  (see \cite{Hu},  \cite{ALW}, \cite{AS}, \cite{LL}) by
$$F_{x_0, t_0}(M):=(4\pi t_0)^{-\frac n2}\int_Me^{-\frac{|x-x_0|^2}{4t_0}}dv.$$
$F$ functional was used by Huisken \cite{Hu} to study mean curvature flow. From definition, $F_{0,1}(M)=(4\pi)^{-\frac n2}\mu(M)$.

The entropy functional  $\eta = \eta(M)$ is defined as the supremum of the $F$  functional over $(x_0, t_0)$:
$$\eta(M):=\sup_{ x_0\in \mathbb{R}^{n+p}, t_0 > 0}F_{x_0, t_0}(M).$$

 Since smooth $n$-dimensional submanifolds are approximated
by an  Euclidean space $\mathbb{R}^n$ on small scales,  $\mathbb{R}^n$ has the least entropy $\eta(\mathbb{R}^{n})=1$ (cf Lemma 7.2 (3) in \cite{CM}).

 On the other hand, in \cite{CM}, Colding-Minicozzi derived the first and second variational formulas for $F$ functional of self-shrinkers in $\mathbb{R}^{n+1}$ with polynomial volume growth. Later, these variational formulas are generalized to the case of higher codimension (\cite{AS}, \cite{ALW}, \cite{LL}). By the variation formulas, we will obtain that the entropy of a self-shrinker with polynomial  volume growth can be achieved by the functional $F_{0,1}$ (cf Subsection $7.2$ \cite{CM}). 
 
 By combining the above facts,  it holds that $\mathbb{R}^n$ has the least weighted volume $\mu(\mathbb{R}^n)=(4\pi)^{\frac n2}$ of all self-shrinkers.  

\qed

In general, complete self-shrinkers do not have the rigidity like gradient shrinking Ricci soliton. Here are some examples.
\begin{example} Self-shrinkers 
$\mathbb{S}^k(\sqrt{2k})\times \mathbb{R}^{n-k}\subset \mathbb{R}^{n+1}$, $0\leq k\leq n$. The first non-zero eigenvalue $\lambda_1=\frac12$ of $\mathcal{L}$  has multiplicity $n$, $n+1$ respectively corresponding to $k=0, 0<k\leq n$.
\end{example}
 
In the following under some hypothesis  on the second fundamental form, we prove that self-shrinkers are properly immersed and give an application  of Theorem   \ref{th-eigen}.  Let $n_{\alpha}, \alpha=1,\ldots, p$ denote a local orthonormal basis on normal bundle of $\Sigma$. Define the map $A^{n_\alpha}: T_qM\to T_qM, q\in M$, $\alpha=1,\ldots, p$   by 
$$ A^{n_\alpha}X=(\overline{\nabla}_Xn_{\alpha})^{\top}, X\in T_qM$$
By a direct computation, Gauss equation implies that
\[\textrm{Ric}_f(X, Y)=\frac 12g(X, Y)-\displaystyle\sum_{\alpha=1}^p\langle A^{n_\alpha}X, A^{n_\alpha}Y\rangle, \quad X, Y\in T_qM.
\]
We obtain  that
\begin{theorem} \label{th-self-1}Let $M^n$ be a complete self-shrinker in $\rr^{n+p}$, $p\geq 1$. If the eigenvalues $\eta_{i,\alpha}$ of $A^{n_\alpha}$, $1\leq i\leq n, 1\leq\alpha\leq p$, satisfy  that $\displaystyle\sup_{1\leq i\leq n}\sum_{\alpha=1}^p\eta_{i,\alpha}^2\leq\delta<\frac 1{2}$ for some nonnegative constant $\delta$,  then 
\begin{itemize}
\item $M$ has finite weighted volume, namely, $\int_{M}e^{-\frac{|x|^2}{4}}dv<\infty$ , polynomial volume growth and  properly immersed.
 \item The spectrum of $\mathcal{L}=\Delta-\frac{1}{2}\langle x,\nabla\cdot\rangle$ on $M$ is 
 discrete and the first nonzero eigenvalue $\lambda_1$ of $\mathcal{L}$ satisfies that $\frac{1-2\delta}{2}\leq \lambda_1\leq \frac12$.

\end{itemize}

\end{theorem}

\begin{proof}
Since all eigenvalues of $A^{n_\alpha}$, $1\leq \alpha\leq p$,  satisfy $\displaystyle\sup_i\sum_{\alpha=1}^p\eta_{i,\alpha}^2\le\delta<\frac 12,$ 
\[\textrm{Ric}_f\geq\frac {1-2\delta}2g.
\]
By \cite{M}, $M$ has finite weighted volume. By \cite{CZ}, $M$ is equivalently properly immersed  and   with polynomial volume growth. 
The rest of conclusion of theorem is from Theorem \ref{th-eigen} and Theorem \ref{th-self-logsob}.

\end{proof}

\noindent In particular, consider  the case that $M$ is an $n$-dimensional self-shrinker in $\mathbb{R}^{n+1}$. Let ${\bf n}$ denote the outward unit normal of $M$.  The Weingarten map or shape operator $A$ is defined by the map $A^{\bf n}$. It holds that 
$$\langle AX, Y\rangle=B(X,Y), \quad X, Y\in T_xM, x\in M.$$
 Theorem \ref{th-self-1} implies that

\begin{cor} \label{cor-self}Let $M^n$ be a complete self-shrinker in $\rr^{n+1}$. If the principle curvatures $\eta_i, i=1,\ldots, n,$ of $M$ satisfy $\displaystyle\sup_{1\leq i\leq n}\eta_i^2\leq\delta<\frac 12$ for some nonnegative constant $\delta$, then 
the same conclusion as in Theorems \ref{th-self-1} holds.
\end{cor}

\section{ Appendix}

 In this section, we give a proof of Proposition \ref{sobo-prop-2}.
We first recall some needed facts in measure theory. Let $(\Omega,\mathcal{F},\mu)$ denote measure space with finite total measure  $\mu(\Omega)$.
Let $L^p(\mu)$ denote the Banach
space of classes of measurable, real-valued functions on $\Omega$, whose $p$-th power is
$\mu$-integrable.

A subset $K$ of $L^1(\mu)$ is called uniformly integrable if given $\epsilon>0$,  there is a $\delta > 0$ so that $\sup\{\int_{E} |f|d\mu : f \in K\} < \epsilon$ whenever $\mu(E) < \delta$.
It is known that 
\begin{lemma}\label{lem-1}(De La Vall$\acute{e}$e Poussin theorem, cf \cite{Me}) Under the above notation, a subset $K$ of $L^1(\mu)$ is uniformly integrable if and only if there exists a non-negative convex function $Q$ with $\displaystyle\lim_{t\to\infty} \frac{Q(t) }{t}=\infty$ so that
$$\sup\left\{\int_{\Omega}Q(|f|)d\mu:f\in K\right\}<\infty.$$
\end{lemma}

In the following let $(M^n,g,e^{-f}dv)$ be a complete smooth metric measure space. We use the same notations as in Section \ref{sec-embed}.    
 
  \begin{remark}  \label{rema-1}  In this remark, we prove that for $(M^n,g,e^{-f}dv)$ with $\mu(M)=\int_Me^{-f}dv<\infty$, if a logarithmic Sobolev  inequality \eqref{ine-ls-16}  holds for all $u\in C_0^{\infty}(M)$  with $\int_Mu^2d\mu=\mu(M)$, then 
it holds for all $u\in H^1(M,\mu)$ with $\int_Mu^2d\mu=\mu(M)$. 

In fact, for  such $u$, there exists a sequence $\{u_k\}$, $u_k\in C_0^{\infty}(M)$ satisfying that $\int_Mu_k^2d\mu=\mu(M)$ and  $u_k\to u$ in $H^1(M,\mu)$. This implies that
$$\int_M|\nabla u_k|^2d\mu\to\int_M|\nabla u|^2d\mu,$$
Since $u_k\to u$ in $H^1(M,\mu)$,  there is a subsequence of $u_k$, still denoted by $u_k$, satisfies $u_k$ a.e.  convegences to $u$.  Note that $t\log t\geq a_0, t\in [0,\infty)$ for some constant $a_0$, and $\mu(M)<\infty$. By Fatou's lemma and \eqref{ine-ls-16},

\begin{align*}
0&\leq  \int_M(u^2\log u^2-a_0)d\mu\\
&\leq  \displaystyle\liminf\int_M(u_k^2\log u_k^2-a_0)d\mu\\
&\leq \displaystyle\liminf \left(2C\int_M|\nabla u_k|^2d\mu+2D\int_Mu_k^2d\mu\right)-\int_M a_0d\mu\\
&= 2C\int_M|\nabla u|^2d\mu+2D\int_Mu^2d\mu- a_0\mu(M)\\
&<\infty.
\end{align*}
So $\int_Mu^2\log u^2d\mu$ exists and 
\begin{align*}
\int_Mu^2\log u^2d\mu\leq 2C\int_M|\nabla u|^2d\mu+2D\int_Mu^2d\mu.
\end{align*}
Hence  \eqref{ine-ls-16} holds for $u\in H^1(M,\mu)$ with $\int_Mu^2=\mu(M)$. 
\end{remark}

Now for $(M,g,e^{-f}dv)$,  assume that  $\mu(M)=\int_Me^{-f}dv$ is finite and  and  the  logarithmic Sobolev inequality \eqref{ine-ls-16} holds on  $(M,g,e^{-f}dv)$ for all $u\in C_0^{\infty}(M)$ satisfying $\int_Mu^2d\mu=\mu(M)$, where $d\mu=e^{-f}dv$. We prove compact embedding of $H^1(M,\mu)$ in $L^2(M,\mu)$.
\medskip

\noindent {\it Proof of Proposition \ref{sobo-prop-2}}.
It is known that the identical map $H^1(M,\mu)\to L^2(M,\mu)$ is an embedding (cf \cite{G}, Section 4.1). So
it suffices to prove that any sequence of $\{u_k\}_{k=1}^{\infty}$ bounded in  $H^1(M,\mu)$-norm has a subsequence converging in  $L^2(M,\mu)$ to a function $u\in L^2(M,\mu)$.  From the standard Sobolev space theory, it is true when $M$ is a compact manifold with or without $C^1$  boundary.
So we only assume that $M$ is noncompact. Let $\{D_i\}$ be an compact exhaustion  of $M$ with $C^1$ boundary $\partial D_i$ for each $i$. It is known that the embedding  $H^1(\Omega_i,\mu)\subset L^2(\Omega_i,\mu)$ is compact. Hence the sequence  $\{u_k\}$, restrict to $\Omega_i$, has a subsequence converging in  $L^2(\Omega_i,\mu)$. Note that an $L^2$ convergence sequence has an a.e. convergent subsequence.   By passing to a diagonal subsequence,  there exists a subsequence of $\{u_k \}$, still denoted by $\{u_k \}$, and a function $u$ defined on $M$  so that $\{u_k\}$ a.e. converges to $u$  on each $D_i$ and hence on $M$.  By Fatou's lemma, $\int_Mu^2d\mu\leq \liminf\int_Mu_k^2<\infty$, that is $u\in L^2(M,\mu)$.

On the other hand, by the hypothesis of proposition:  the  logarithmic Sobolev inequality \eqref{ine-ls-15} holds for  $H^1(M,\mu)$, all $u_k$ satisfy
\begin{align}\label{ine-ls-9}
&\int_Mu_k^2\log (u_k^2)d\mu-\left(\int_Mu_k^2d\mu\right)\log\left(\frac {\int_Mu_k^2d\mu}{\mu(M)}\right)\nonumber\\
&\leq 2C\int_M|\nabla u_k|^2d\mu+2D\int_Mu_k^2d\mu.
\end{align}
\eqref{ine-ls-9}  together with the   boundedness of $H^1(M,\mu)$-norm of $u_k$ implies that  there exists a constant $\bar{C}$ satisfying
\[\int_Mu_k^2\log u_k^2d\mu\le \bar{C}.
\]
 Take $Q(t)=t\log t-a_0\geq 0, t\in [0,\infty)$, for some constant $a_0>0$. One can see that $Q(t)$ and $\{u_k^2\}$ satisfy the conditions of Lemma \ref{lem-1}  (De La Vall$\acute{e}$e Poussin theorem) and thus $\{u_k^2\}$ is uniformly integrable.

 Now with the  facts of  a.e. convergence of $\{u_k\}$ to $u$ and the uniform integrability of $\{u_k^2\}$,  by a standard argument using Egorov's theorem, similar to the proof of Vitali convergence theorem, one can prove that $\int_M|u_k-u|^2\to 0$, that is, $u_k\to u$ in $L^2(M,\mu)$. Therefore the embedding $H^1(M,\mu)\hookrightarrow L^2(M,\mu)$ is compact.
 By the standard theory  in PDE (cf \cite{G} Theorem 10.20), the compact embedding of $H^1(M,\mu)$ is equivalent to the discreteness of spectrum of $\Delta_f$.
  Thus we complete the proof.

\qed

\end{document}